\documentclass[12pt]{article}

\usepackage{epsfig}
\include{psfig}
\input epsf

\usepackage{amssymb}
\usepackage{amsmath}
\usepackage{amsthm}

\newtheorem{theorem}{Theorem}[section]

\newtheorem{lemma}{Lemma}[section]
\newtheorem{corollary}{Corollary}[section]
\newtheorem{conjecture}{Conjecture}[section]

\newenvironment{example}{\refstepcounter{theorem}\noindent%
\par\addvspace{\medskipamount}\pagebreak[2]\noindent%
{\bf Example \thetheorem\
}\ignorespaces}{\par\addvspace{\bigskipamount}}

\newcommand{\StartV}{
\left( \begin{array}{c}
}
\newcommand{\EndV}{
\end{array} \right)
}

\title{{\bf On the distances between Latin squares and the smallest defining set size}}

\author{
Nicholas Cavenagh and Reshma Ramadurai \\
 Department of Mathematics \\
The University of Waikato \\
Private Bag 3105, Hamilton, New Zealand \\
{\texttt{nickc@waikato.ac.nz}} \\
{\texttt{reshmar@waikato.ac.nz}} \\
}

\date{}

\begin{document}

\maketitle\thispagestyle{empty}
\def\baselinestretch{1.0}\small\normalsize
\def \l {\lambda}
\def \mod {\mbox{ mod }}
\parskip=2mm

\begin{abstract}
In this note we show that for each Latin square $L$ of order $n\geq 2$, there exists a Latin square $L'\neq L$ of order $n$ such that $L$ and $L'$ differ in at most $8\sqrt{n}$ cells. Equivalently, each Latin square of order $n$ contains a Latin trade of size 
at most $8\sqrt{n}$. We also show that the size of the smallest defining set in a Latin square is $\Omega(n^{3/2})$. 
%That is, there are constants $c$ and $n_0$ such that for any $n>n_0$ the size of the smallest defining 
%set of order $n$ is at least $cn^{3/2}$.  

\vspace{2mm}
{\it Keywords: Latin square, Latin trade, defining set, critical set, Hamming distance.} 
\end{abstract}

\section{Introduction}\label{sec1}

For each positive integer $a$, we use the notation $N(a)$ for the
set of integers $\{0,1,2,\dots ,a-1\}$.  

A {\em partial Latin square} of order $n$ is an $n\times n$
array, where each cell of the array is either empty or contains a 
 symbol from $N(n)$, such that each symbol occurs at
at most once per row and at most once per column. 
A {\em Latin square} is a partial Latin square in which
no cell is empty, and hence each symbol
occurs precisely once in each row and once in each column.

Indexing rows and columns by $N(n)$, 
we may consider a partial Latin square
to also be a set of ordered triples of the form $(i,j,L(i,j))$, where 
$L(i,j)$ is the symbol in row $i$ and column $j$ (if occupied). 
The {\em distance} (or {\em Hamming distance} \cite{Dr1}) between two partial Latin squares $L$ and $L'$ of the same order is then 
defined to be $|L\setminus L'|$.  

We show the following.
\begin{theorem}
\label{main}
For each Latin square $L$ of order $n$, there exists a Latin square $L'\neq L$ of order $n$ such that $|L\setminus L'|\leq 8\sqrt{n}$. 
\end{theorem}

Note that it is trivial to obtain an upper bound of $2n$ in the above; simply swap two rows of $L$ to create $L'$. 
Theorem \ref{main} however is the first such upper bound which is $o(n)$.  This a step towards the possible truth of Conjecture 4.25 from \cite{Cavthesis}: 
\begin{conjecture}
\label{loggy}
For each Latin square $L$ of order $n$,
$$\hbox{{\rm min}}\{|L\setminus L'| \mid \hbox{ $L'$ is a Latin square of order $n$ and $L'\neq L$}\}=O(\log{n}).$$
\end{conjecture}

We may also state Theorem \ref{main} as a result about {\em Latin trades}. 
Given two distinct Latin squares $L$ and $L'$ of the same order $n$, 
$L\setminus L'$ is said to be a {\em Latin trade} with disjoint mate $L'\setminus L$. In terms of arrays, 
we say that two partial Latin squares are {\em row balanced} if corresponding rows contain the same set of symbols; {\em column balanced} is defined similarly. 
A Latin trade $T$ and its disjoint mate $T'$ are thus a pair of partial Latin squares which occupy the same set of cells, are disjoint and are both row and column balanced. 

\begin{example}
In the example below, $d(L_1,L_2)=18$ and $L\setminus L'$ is a Latin trade with disjoint mate $L'\setminus L$. 
$$
\begin{array}{cc}
\begin{array}
{|c|c|c|c|c|c|c|}
\hline
1 & 2 & 0 & 6 & 3 & 4 & 5 \\
\hline
6 & 1 & 5 & 4 & 0 & 2 & 3 \\
\hline
0 & 5 & 4 & 2 & 1 & 3 & 6 \\
\hline
3 & 4 & 2 & 1 & 5 & 6 & 0 \\
\hline
4 & 3 & 1 & 5 & 6 & 0 & 2\\
\hline
2 & 6 & 3 & 0 & 4 & 5 & 1 \\
\hline
5 & 0 & 6 & 3 & 2 & 1 & 4 \\
\hline
\end{array}
 &
\begin{array}
{|c|c|c|c|c|c|c|}
\hline
2 & 1 & 0 & 6 & 3 & 4 & 5 \\
\hline
6 & 5 & 4 & 1 & 0 & 2 & 3 \\
\hline
0 & 2 & 5 & 4 & 1 & 3 & 6 \\
\hline
3 & 4 & 2 & 5 & 6 & 0 & 1 \\
\hline
4 & 3 & 1 & 2 & 5 & 6 & 0 \\
\hline
1 & 6 & 3 & 0 & 4 & 5 & 2 \\
\hline
5 & 0 & 6 & 3 & 2 & 1 & 4 \\
\hline
\end{array}\\
L_1 & L_2 \\
 & \\
\begin{array}
{|c|c|c|c|c|c|c|}
\hline
1 & 2 &  &  &  &  &  \\
\hline
 & 1 & 5 & 4 &  &  &  \\
\hline
 & 5 & 4 & 2 & &  &  \\
\hline
 &  &  & 1 & 5 & 6 & 0 \\
\hline
 &  &  & 5 & 6 & 0 & 2\\
\hline
2 &  &  &  &  &  & 1 \\
\hline
&&&&&& \\
\hline
\end{array}
 &
\begin{array}
{|c|c|c|c|c|c|c|}
\hline
2 & 1 &  &  &  &  &  \\
\hline
 & 5 & 4 & 1 &  &  &  \\
\hline
 & 2 & 5 & 4 &  &  &  \\
\hline
 &  &  & 5 & 6 & 0 & 1 \\
\hline
 &  &  & 2 & 5 & 6 & 0 \\
\hline
1 &  &  &  &  &  & 2 \\
\hline
&&&&&& \\
\hline
\end{array}\\
L_1\setminus L_2 & L_2\setminus L_1 \\
\end{array}$$
\label{eggsy}
\end{example}

Theorem \ref{main} thus implies: 
\begin{theorem}
\label{main2}
Each Latin square $L$ of order $n$ contains a Latin trade $T$ such that $|T|\leq 8\sqrt{n}$. 
\end{theorem}

As Latin squares are precisely operation tables for quasigroups, our main result can also be considered in the context of Hamming distances of algebraic objects (see \cite{Dr1} for more detail on this topic). 

Let $B_n$ be the Latin square formed by the addition table for the integers modulo $n$; (i.e. $B_n(i,j)=i+j \mod{n}$). 
The upper bound in Conjecture \ref{loggy} cannot be decreased, since it is 
known that any Latin trade in $B_n$ has size at least $e\log{p}+3$, where $p$ is the least prime that divides $n$ (\cite{DrKe, Cav2}). 
It was recently shown in \cite{Sz} that for each integer $n$, $B_n$ contains a Latin trade of size $5\log_2{n}$. 

If Conjecture \ref{loggy} above is true, it thus may be that the back circulant Latin square is the ``loneliest'' of all Latin squares; i.e. the Latin square with greatest minimum distance to any other Latin square. 
The smallest Latin trade, known as an {\em intercalate}, has size $4$ and consists of a $2\times 2$ subarray on two symbols (the disjoint mate is formed by swapping the symbols). 
It is shown in \cite{McW}  that for any $\epsilon >0$, almost all Latin squares of order $n$ possess at least $O(n^{3/2-\epsilon})$ intercalates. 
Thus we know that most Latin squares are not as ``lonely'' as the back circulant Latin square.   

A {\em defining set} $L'$ for a Latin square $L$ of order $n$ is a subset $L'\subseteq L$ such that if $L''$ is a Latin
square of order $n$ and $L'\subseteq L''$ then $L''=L$. In other words, a defining set has unique completion to a Latin square of specified order.   
If $T$ is a Latin trade in a Latin square $L$ with disjoint mate $T'$, $(L\setminus T)\cup T'$ is a Latin square distinct from 
$L$. The following is immediate.  
\begin{lemma}
If $D$ is a defining set for a Latin square $L$ and $T$ is a Latin trade such that $T\subseteq L$, then $T\cap D\neq \emptyset$. 
\label{easy}
\end{lemma}
It comes as no surprise then, that the new results on Latin trades in this paper yield a new result on defining sets. 

A much studied open problem is to determine the smallest possible size of a defining 
set of order $n$, denoted by scs$(n)$. It is conjectured that the correct value
for scs$(n)$ is equal to $\lfloor n^2/4\rfloor$. (This has been verified computationally for $n\leq 8$ \cite{Be}).
Defining sets of such size are known to exist for each $n\geq 1$ (\cite{CDS,DC}).
Currently the best known lower bound for large $n$ is scs$(n)\geq n\lfloor (\log{n})^{1/3}/2\rfloor$ \cite{Cav3}. 
This in turn improved previous results given in \cite{FFR} and \cite{HAF}.
Using the Latin trades constructed in Section 2 of this paper, we improve this bound to the following. 

\begin{theorem}
A defining set in a Latin square of order $n$ has size $\Omega(n^{3/2})$. 
\label{main3}
\end{theorem}

A {\em critical set} for a Latin square of order $n$ is a minimal defining set; i.e. a defining set $D$ which is not the superset of any smaller defining set of the same order. Clearly scs$(n)$ also gives the size of the smallest possible critical set of order $n$. 

\section{Latin trades}

Given a Latin square $L$ of order $n$, two distinct symbols $a,b\in N(n)$ and
some function $f=f(n)\in {\mathbb N}$ we later specify, 
 we construct a coloured digraph 
$G(=G_{L,a,b,f})$ of order $n$ as follows. 
 The vertices of $G$ are labelled
with $N(n)$ and correspond to the columns of $L$. 
Each directed edge will be coloured green, black or blue.  
In what follows, a {\em directed cycle} of length $m\geq 1$ is a set of directed edges (of any colour) of the form 
$\{[v_i,v_{i+1})\mid i\in N(m)\}$ where $v_m=v_0$ and $v_1,v_2,\dots ,v_{m-1}$ are distinct vertices in $G$. Note that loops and circuits of length $2$ are included in this definition.  

Whenever $(r,c,a),(r,c',b)\in L$ for some fixed row $r$, 
we add a green edge from $c$ to $c'$ in the digraph $G$. There are no other green edges.  
We say that $\{(r,c,a),(r,c',b)\}$ is the partial Latin square {\em associated} with this green edge.

From the definition of a Latin square, the following is immediate.
\begin{lemma}
The green edges of the graph $G$ form a directed $2$-factor of $G$. 
\end{lemma}

Next, for each $(r,c,a)\in L$ and $r'\neq r$ ($\ast$), 
we define the following (infinite) sequence. 
Let $c_0=c$ and $e_0=a$. For each $i\geq 0$, let  
$e_{i+1}$ be the entry such that $(r',c_i,e_{i+1})\in L$ and let 
$c_{i+1}$ be the column such that $(r,c_{i+1},e_{i+1})\in L$. 
This creates a ``zig-zag'' pattern, as shown below:

$$\begin{array}{r|c|c|c|c|}
\multicolumn{1}{r}{} & 
\multicolumn{1}{c}{c_0} & 
\multicolumn{1}{c}{c_1} & 
\multicolumn{1}{c}{c_2} & 
\multicolumn{1}{c}{} 
\\
\cline{2-5}
r & e_0 & e_1 & e_2 & \dots \\
\cline{2-5}
r' & e_1 & e_2 & e_3 & \dots \\
\cline{2-5}
\end{array}$$

For each integer $k$ we define $P_k(r,r',c)$ to be the following subset of $L$ of size $2k$:  
$$\{(r,c_i,e_i),(r',c_i,e_{i+1})\mid i\in N(k)\}.$$  
By finiteness, observe that $e_K=a$ for some $K$ such that $0<K\leq n$; assume $K(=K(r,r',c))$ is minimum with respect to this property. 
\begin{lemma}
The partial Latin square $P_K(r,r',c)$ is a Latin trade of size $2K$. 
\label{rowcycle}
\end{lemma}

\begin{proof}
Simply swap $r$ and $r'$ in each triple to form the disjoint mate. 
\end{proof}

Any such Latin trade as in the previous lemma is called a {\em row cycle trade}. 

We are now ready to define the black edges in the digraph $G$. 
Suppose that $e_k=b$ for some $k$ such that $0<k<K$ and $k\leq f(n)$. 
Then add a black edge from $c_0$ to $c_{k-1}$ in the graph $G$. 
We say that $P_k(r,r',c=c_0)$ is the partial Latin square {\em associated} with 
a black edge. 

Note that not every choice of $(r,c,a)$ and $r'\neq r$ (see $\ast$)  
will result in a black edge; there are two possible obstacles. 
Firstly, it may happen that the sequence $e_0,e_1,e_2,\dots $ does not contain the symbol $b$; equivalently, the Latin trade 
$P_K(r,r',c)$ does not include symbol $b$. Secondly, it is possible that $e_k=b$ implies that $k>f(n)$. 

The following lemma is straightforward.

\begin{lemma}
Considering only green and black edges, the graph $G$ has neither loops nor multiple edges (i.e. no two edges sharing the same initial vertex $c$ and the same terminal vertex $c'$.)
\label{simple}
\end{lemma}

It is our next aim to choose $a$ and $b$ in order to maximize the number of black edges in our digraph. 

\begin{lemma} 
Either the Latin square $L$ contains a Latin trade of size at most 
$2f(n)$ {\em or} the total number of black edges in digraphs of the 
form $G_{L,a,b,f}$ (where $a$ and $b$ are distinct symbols) is at 
least $n^2(n-1)(f(n)-1)$.  
\end{lemma}

\begin{proof} 
If the Latin square $L$ contains a Latin trade of size at most $2f(n)$ we are done. 
So in what follows, we assume that no such Latin trade exists.  

Now, there are $n^2(n-1)$ ways of choosing an element $(r,c,a)\in L$ and 
a row $r'\neq r$. Each such choice yields lists $c_0,c_1,\dots c_{f(n)-1}$ and 
$e_0,e_1,\dots ,e_{f(n)-1}$ as above. 
If $f(n)\geq K$, then from Lemma \ref{rowcycle} there exists a Latin trade
of size $2K\leq 2f(n)$, a contradiction. 
Thus $f(n)< K$ and   
 the list 
$c_0,c_1,\dots ,c_{f(n)-1}$ has no repeated elements. 
Moreover, for each $i$ such that $0<i<f(n)$, there is a black edge 
from $c_0$ to $c_i$ in the graph $G_{L,e_0,e_{i+1},f(n)}$. 

Thus there are a total of $n^2(n-1)(f(n)-1)$ black edges 
in all of the graphs of the form $G_{L,a,b,f(n)}$ (where $a\neq b$).
Note there is no over-counting here because the edges are directed and for each 
choice of $a$ in column $c$ and $b$ in column $c'$ there are unique rows $r$ and $r'$ such that $(r,c,a),(r',c',b)\in L$. 
\end{proof}

Since there are $n(n-1)$ choices for the ordered pair of symbols $(a,b)$ (where $a\neq b$), the following corollary is immediate. 

\begin{corollary}
There exist distinct symbols $a$ and $b$ such that 
the graph $G_{L,a,b,f}$ contains at least $n(f(n)-1)$ black edges. 
\label{blacky}
\end{corollary}

%Without loss of generality we henceforth assume the graph $G=G_{L,0,1,f}$ has at least $n(f(n)-1)$ black edges. 

We finally define blue edges as follows. Whenever there is a green edge from 
$c_1$ to $c_2$, a black edge from $c_3$ to $c_2$ and a green edge from $c_3$ to $c_4$ (and $c_1\neq c_4$), we add a blue edge from $c_1$ to $c_4$. 
By Lemma \ref{simple}, $c_1\neq c_2$, $c_1\neq c_3$, $c_2\neq c_3$, $c_2\neq c_4$ and $c_3\neq c_4$. 
However $c_1=c_4$ is possible; in such a case our blue edge is a loop. 
The partial Latin square {\em associated} with a blue edge is the union of the partial Latin squares associated with these two green edges and one black edge. 

Our next aim is to show that any directed cycle in $G$ gives rise to a Latin trade in $L$. 

A {\em symbol cycle trade} is a Latin trade containing only two symbols. 
For example:
$$\begin{array}{|c|c|c|c|c|}
\hline 
a & b & & & \\
\hline 
 & a & b & & \\
\hline 
 &  & a & & b \\
\hline 
 &  &  & b & a \\
\hline 
b &  &  & a &  \\
\hline 
\end{array}$$
Clearly the disjoint mate is formed by swapping symbols $a$ and $b$ within each row. 
In our graph, such a trade corresponds to a coherently directed cycle of green edges. 
The following lemma is immediate.
\begin{lemma}
Let $T$ be the union of partial Latin squares associated with the edges of a green directed cycle of length $k$. Then $T$ is a Latin trade of size $2k$. 
\label{greeny}
\end{lemma}

We now wish to consider directed cycles with edges of colours green, black or blue. 
For expediency we say that a partial Latin square associated with a directed edge of colour $s$ is {\em coloured $s$}.
For each coloured partial Latin square $P$, we 
define the {\em mate} $P'$ of $P$ as follows.  
(The mates of partial Latin squares will ultimately define a disjoint mate of a Latin trade.) 
Firstly, the mate $P'$ of $P$ has the same set of occupied cells of $P$ but is disjoint from $P$. 
For a green partial Latin square $P=\{(r,c,a),(r,c,b)\}$, its mate 
$P'=\{(r,c,b),(r,c,a)\}$. 
For a black partial Latin square $P$, the mate $P'$ of $P$ 
is formed by swapping the rows of $P$ and then the symbols $a$ and $b$. 
For a blue partial Latin square $P$, the mate $P'$ of $P$ is formed by swapping the symbols in each column with two symbols and 
swapping $a$ with $b$ in columns with only one symbol.

Examples of mates for green, black and blue partial Latin squares are exhibited below, with symbols from $P'$ given as subscripts. 

$$
\begin{array}{ccc}
\begin{array}{|c|c|}
\hline
a_b & b_a \\
\hline
\end{array}
& 
\begin{array}{|c|c|c|c|}
\hline
a_1 & 1_2 & 2_3 & 3_a \\
\hline
1_b & 2_1 & 3_2 & b_3 \\
\hline
\end{array}
& 
\begin{array}{|c|c|c|c|c|}
\hline
a_b & b_2 & 2_1 & 1_a & \\
\hline
& 2_b & 1_2 & a_1 & b_a \\
\hline
\end{array}\\
\mbox{Green} & 
\mbox{Black} & 
\mbox{Blue} \\
\end{array}$$

Observe the following lemma. 

\begin{lemma}
If $P'$ is the mate of a partial Latin square $P$ associated with a green, black or blue edge, then:
\begin{itemize}
\item $P$ and $P'$ are row-balanced; 
\item If we remove $a$ from $P$ and $b$ from $P'$ in the first column and if we remove $b$ from $P$ and $a$ from $P'$ in the final column then $P$ and $P'$ become column-balanced. 
\end{itemize}
(Here the first and final columns correspond to the initial and terminal vertices, respectively, of the coloured edge.)
\end{lemma}

What we have then, are partial Latin squares paired with mates that behave almost like Latin trades except for at the initial and terminal vertices. 
However if each terminal vertex of one edge coincides with an initial vertex of 
another edge (that is, we have a directed cycle) {\em and} if the partial Latin squares are disjoint, it is clear we have a Latin trade. 
(This is demonstrated in Example \ref{eggsy}: there is a green edge from the first column to the second column, a black edge from the second to the fourth column, another black edge from the fourth to the seventh column and finally a green edge from the seventh to the first column.)  This is true even if our directed cycle has one edge; i.e. is a blue loop. 

\begin{corollary}
Suppose there is a directed cycle $C$ in $G$ (with edges of any colour) and that the partial Latin squares associated with the edges of $C$ are pairwise disjoint.  
Then there is a Latin trade of size at most $2g+2bf(n)+2y(f(n)+1)$, where $g$, $b$ and $y$ are the number of green, black and blue edges, respectively, in the cycle $C$. 
\end{corollary}

We are almost there - we just have to consider the case when the associated partial Latin squares are not disjoint.

\begin{theorem}
Let $C$ be a directed cycle of minimum length in $G$.   
Then there is a Latin trade of size at most $2g+2bf(n)+2y(f(n)+1)$, where $g$, $b$ and $y$ are the number of green, black and blue edges, respectively, in the cycle $C$. 
\label{coloury}
\end{theorem}

\begin{proof}
Let $C$ be a directed cycle in $G$ and consider only the partial Latin squares which are associated with edges of $G$. 
Firstly, if a green partial Latin square intersects either a black or blue partial Latin square, two directed edges have the same initial or terminal vertex, a contradiction. 

Next suppose two black partial Latin squares intersect; say 
$P_k(r_1,r_2,c)$ 
and $P_{\ell}(r_3,r_4,c')$. 
Then clearly $\{r_1,r_2\}\cap \{r_3,r_4\}\neq \emptyset$. 
If either $r_1=r_3$ or $r_2=r_4$ then two directed edges in $C$ start at the same vertex or terminate at the same vertex (respectively), a contradiction. 
Otherwise if  $r_1=r_4$ and $r_2=r_3$,  
then 
$P_k(r_1,r_2,c)\cup  
P_l(r_3,r_4,c')$ forms a row cycle Latin trade of size $2(k+\ell)$ and we are done. 

Next, suppose that $r_3=r_2$ and $r_1\neq r_4$ and 
the associated partial Latin squares intersect. (The case when $r_1=r_4$ and $r_2\neq r_3$ is equivalent.)
Removing the two black edges from $C$ creates two directed paths; let $D$ be the directed path which starts at $c_k$ (i.e. the final column of $P_k(r_1,r_2,c)$) and terminates at 
$c'$. However since $r_3=r_2$ there is a green edge from $c'$ to $c_k$; together these form a directed cycle with length shorter than $C$, a contradiction. 

Next suppose two 
blue partial Latin squares intersect.
Let the black partial Latin squares which are subsets of these blue partial Latin squares be 
$P_k(r_1,r_2,c)$ 
and $P_l(r_3,r_4,c')$, respectively. 
The cases $r_1=r_3$ or $r_2=r_4$ lead to contradictions as above. 
The case $r_1=r_4$ and $r_2=r_3$ implies a trade of size $2(k+\ell)$, similarly to above. 
This leaves the case $r_2=r_3$ and $r_1\neq r_4$ (equivalent to the 
case $r_1=r_4$ and $r_2\neq r_3$). Again, there is a green edge from $c'$ to $c_k$, which 
combined with a directed path from $C$, forms a directed cycle which is shorter than $C$.

Finally, suppose a black partial Latin square 
$P_k(r_1,r_2,c)$ 
intersects with a blue partial Latin square 
(containing the black partial Latin square $P_{\ell}(r_3,r_4,c')$). 
If $P_k(r_1,r_2,c)= 
P_{\ell}(r_3,r_4,c')$, then there is a green edge from $c_k$ to $c'$, which   
combined with a directed path from $C$ forms a directed cycle which is shorter than $C$.
Otherwise if $\{r_1,r_2\}=\{r_3,r_4\}$ there is a trade of size $2(k+\ell)$ within these rows. 
The cases $r_1=r_3$ or $r_2=r_4$ lead to contradictions as above. 
If $r_1=r_4$ and $r_2\neq r_3$, 
let $c''$ be the column of $P_{\ell}(r_3,r_4,c')$ which contains $b$ in row $r_4$.
Then there is a directed path in $C$ from $c''$ to $c$ 
and a green edge from $c$ to $c''$, creating a cycle shorter than $C$. 
Finally, if $r_2=r_3$ and $r_1\neq r_4$, 
let $c'''$ be the column of $P{\ell}(r_3,r_4,c')$ which contains $b$ in row $r_3$.
Then there is a directed path in $C$ from $c'''$ to $c'$ and a green edge from $c'$ to $c'''$, again causing a contradiction.  
\end{proof}

\section{An upper bound on the distance between Latin squares}

In this section we give a proof of Theorem \ref{main}. 
To ultimately show the existence of Latin trades, we first make some observations about drawing edges 
between parallel line paths without edges crossing. 
In what follows, for the sake of simplicity of explanation, 
we embed two directed paths $P$ and $Q$ (of length $p$ and $q$ respectively, where length is the number of vertices) in
the Euclidean plane so that $P$ lies entirely within the line $y=0$ and $Q$ lies entirely within the line $y=1$ 
and the vertices have integer coordinates with each edge directed from left to right. (Really we simply need $P$ and $Q$ to be drawn as parallel line segments, the above specificity avoids any ambiguity). 
The following lemma is easy to show for example by induction.  

\begin{lemma}
At most $p+q-1$ straight line segments can be drawn between vertices of $P$ and vertices of $Q$ without any edges crossing. 
\end{lemma}

For our purposes we need something more specific. 

\begin{lemma}
Let $p,q\geq 3$. 
If more than $2(p+q-2)$ straight line segments are drawn between vertices of $P$ and vertices of $Q$, then 
there exists two edges which cross
such that the edges do not use vertices adjacent in $P$ or adjacent in $Q$. 
\end{lemma}

\begin{proof}
Properly colour the vertices of $P$ with colours $c_1$ and $c_2$ and the vertices of $Q$ with colours $c_1'$ and $c_2'$.
(That is, within each path vertices of the same colour are not adjacent.) 
Consider the line segments between vertices of colours $c_i$ and $c_j'$ for some fixed $i$ and $j$, 
$i,j\in\{1,2\}$. If two such edges cross, these edges do not use vertices adjacent in $P$ nor do they use vertices adjacent in $Q$. 
Since there are four choices for $i$ and $j$, from the previous lemma 
we can draw at most $2(p+q-2)$ line segments without any edges crossing. 
\end{proof}

We now explain why we need the previous lemmas.

\begin{lemma}
Let $P$ and $Q$ be directed paths of green edges embedded in the Euclidean plane as above, where $p\geq 3$ and $q\geq 3$. 
If there exists more than $2(p+q-2)$ black edges between vertices of $P$ and vertices of $Q$, 
then there exists a directed cycle in $G$ on the vertices of $P$ and $Q$ such that
the number of edges which are either black or blue is at most $2$.
\label{kross} 
\end{lemma}

\begin{proof}
Let the vertices of $P$ be $1,2,\dots ,p$ and let the vertices of $Q$ be $1'$, $2',\dots ,q'$ where 
each directed edge is from $i$ to $i+1$ in $P$ or from $i'$ to $(i+1)'$ in $Q$ for some $i$.
From the previous lemma, there exists a 
black edge on vertices $i$ and $(j+\ell)'$ and a black edge on vertices $j'$ and $i+k$ where $k,\ell\geq 2$.
If these black edges are directed from $(j+\ell)'$ to $i$ and from $i+k$ to $j'$, respectively, we are done.
If there is a black edge from $i$ to $(j+\ell)'$, then by definition there is a blue edge from 
$(j+\ell-1)'$ to $i+1$. 
If there is a black edge from $j'$ to $(i+k)$, then by definition there is a blue edge from 
$i+k-1$ to $(j+1)'$. 
In any case, we can construct the required directed cycle. 
\end{proof}

\begin{theorem}
For $n\geq 2$, each Latin square contains a Latin trade of size at most $8\sqrt{n}$.
\end{theorem}

\begin{proof}
If $n< 16$, $8\sqrt{n}> 2n$ and since any Latin square of order $n\geq 2$ has a Latin trade of size $2n$ we may assume henceforth that $n\geq 16$.
Let $b=4$, $k=4/3$ and $d=19/6$. 
Suppose, for the sake of contradiction, there exists a Latin square $L$ of order $n$ such that every Latin trade 
in $L$ has size greater than $2b\sqrt{n}$. 
We consider the directed coloured graph $G=G_{L,a,b,f(n)}$ as defined in the previous section
where $f(n)=\lceil d\sqrt{n}\rceil +1$. 
From Lemma \ref{greeny}, each directed green cycle in $G$ has length greater than $b\sqrt{n}\geq 16$. 
We now partition the green edges into directed paths so that each path has length (number of vertices) at most $k\sqrt{n}$. It is clear that we can minimize the number of such paths by ensuring that each cycle contributes at most one path of length less than $\lfloor k\sqrt{n}\rfloor$.

Since each cycle has length greater than $b\sqrt{n}$, the number of directed green cycles is less than $\sqrt{n}/b$. 
Thus the number of paths  of length less than $\lfloor k\sqrt{n}\rfloor$ is at most $\sqrt{n}/b$. 
Also the total number of paths of length equal to $\lfloor k\sqrt{n}\rfloor$ is at most $\sqrt{n}/k$. 
Thus the total number of paths of length in our partition of the green edges is at most $\sqrt{n}(b+k)/bk$. 
In turn, the total number of pairs of paths is less than $n(b+k)^2/2b^2k^2$. 

By Corollary \ref{blacky}, $G$ has at least $dn^{3/2}$ black edges.  
Suppose there are at least $4k\sqrt{n}$ black edges between 
two of the paths of length at least $3$ in the partition. Then by Lemma \ref{kross}, there is 
a directed cycle in $G$ using at most $2k\sqrt{n}$ green edges and at most $2$ edges which are either black or blue. 
Thus by Theorem \ref{coloury} there is a Latin trade of size at most 
$2(k-1)\sqrt{n}+2d\sqrt{n}+4=2\sqrt{n}(k+d-1)+4$. 
Since $b= k+d-1/2$ and $n\geq 16$, we are done in this case. 

Thus there are less than $4k\sqrt{n}$ black edges between each pair of directed green paths of length at least $3$. 
There are also at most $4k\sqrt{n}$ black edges between a path of length at most $2$ and any other path. 
If there is a black edge using two vertices from the same directed green path, then
 we create a directed cycle with at most $k\sqrt{n}$ edges and thus we are done.

Thus the total number of black edges is less than  
$4k\sqrt{n}$ times the number of pairs of directed paths. 
Given the above lower bound of $dn^{3/2}$ on the number of black edges, we have:
$$\frac{2(b+k)^2}{b^2k}\geq d,$$
a contradiction given the above values of $b$, $d$ and $k$. 
\end{proof}

\section{A lower bound on the size of a defining set}

In this section we give a proof of Theorem \ref{main3}. To that end, it suffices to prove the following.

\begin{theorem}
The size of the defining set of any Latin square of order $n$ is $\Omega(n^{3/2})$. 
\label{definy}
\end{theorem}

\begin{proof}
Suppose for the sake of contradiction, there exists a Latin square $L$ of order $n$ with a defining set $D'$ such that 
$|D'|\leq cn^{3/2}$ where $0<c<1/\sqrt{40}$. If we can show that $L\setminus D'$ contains a Latin trade, we are done by Lemma \ref{easy}.
It thus suffices to show that $L\setminus D$ contains a Latin trade for any $D$ such that $D'\subseteq D\subset L$ and 
$|D|=\lceil cn^{3/2}\rceil$. 
We define graphs $G_{L,a,b,f(n)=n}$ with coloured edges as in Section 2, with the proviso that only edges corresponding to partial Latin squares which do {\em not} include elements of $D$ are included. 
By Theorem \ref{coloury} it suffices to show that the graph $G$ contains a coloured cycle for some ordered pair of symbols $(a,b)$.  
(We set $f(n)=n$ as we simply need to show the existence of a Latin trade; its size to us is irrelevant.) 
We let ${\mathcal B}$ be the total number of black edges in any graph of the form $G_{L,a,b,n}$.  

%; that is 
%$${\mathcal G}=\bigcup_{a\neq b} G_{L,a,b,n}.$$ 
%Note that while each graph $G_{L,a,b,n}$ is simple, the graph ${\mathcal G}$ may have parallel edges. 

We first obtain a lower bound for ${\mathcal B}$. 
Let $x_i$ be the number of elements in row $i$ of $D$. 
Then $\sum_{i=0}^{n-1} x_i= \lceil cn^{3/2}\rceil$.   
Consider the first two rows of $L$. Let $k=x_0+x_1$. 
The elements in these rows partition into row-cycles (or possibly just one large row-cycle), as in Lemma \ref{rowcycle}. 
 Rearrange the columns so that any columns in the same row-cycle form a consecutive set of integers. 
Since each row cycle gives a Latin trade, each row cycle contains an element of $D$. Rearrange the columns further so that the first column in each row-cycle contains an element of $D$. Finally rearrange the columns within each trade so that the element in 
cell $(1,c)$ is the same as the element in cell $(0,c+1)$, unless $c+1$ belongs to a different row-cycle to $c$. 

Let $c_1$, $c_2,\dots ,c_e$ be the columns containing elements of $D$ where $k/2\leq e\leq k$. (The lower bound arises in the extreme case that each column contains two elements from $D$ in rows $1$ and $2$). For each $1\leq i< e$ we define block $B_i$ to be the set of columns strictly between $c_i$ and $c_{i+1}$ (not including $c_i$ and $c_{i+1}$). Let $b_i=c_{i+1}-c_i-1$ and $b_e=(n-1)-c_e$. Observe that $\sum_{i=1}^e b_i= n-e\geq n-k$. 

We illustrate the above notation in the example below. Elements of $D$ are given in bold and underlined; $x_0=x_1=2$, $k=4$, $e=3$, $c_1=0$, $c_2=3$, $c_3=5$, $b_1=2$, $b_2=1$ and $b_3=4$. 
$$\begin{array}{|c|c|c|c|c|c|c|c|c|c|}
\hline
{\bf \underline{0}} & 7 & 2 & 3 & 4 & {\bf \underline{5}} & 6 & 1 & 8 & 9 \\ 
\hline
7 & 2 & 0 & {\bf \underline{4}} & 5 & {\bf \underline{6}} & 1 & 8 & 9 & 3 \\
\hline
\end{array}$$

Observe that any pair of distinct columns within a block gives rise to two black edges. 
For example, in the above the columns in $B_1$ give rise to an edge from the second to the third column in 
$G_{L,7,0,n}$ and an edge from the third to the second column in $G_{L,0,7,n}$.   
Thus the number of black edges arising from rows $0$ and $1$ is equal to
\begin{eqnarray*}
2\sum_{i=1}^e \binom{b_i}{2} & = & e-n+ \sum_{i=1}^e b_i^2  \\
& \geq &  e-n+ e\left((n-e)/e\right)^2 \\
& = & 2e - 3n + n^2/e \\
& \geq & k - 3n  + n^2/k \\
& = & (x_0+x_1) - 3n + \frac{n^2}{(x_0+x_1)}.
\end{eqnarray*}

So, considering all pairs of rows,
$${\mathcal B}\geq c(n-1)n^{3/2} -\frac{3}{2}n^2(n-1)+  n^2\sum_{i< j} \frac{1}{x_i+x_j}.$$  
The last sum in the above expression is minimized when all $x_i$'s are equal, i.e., $x_i = \lceil cn^{3/2}\rceil/n \le cn^{1/2} + 1$, so:
\begin{equation}\label{blackylb}
{\mathcal B}\geq c(n-1)n^{3/2} -\frac{3}{2}n^2(n-1)+  \frac{n^{3}(n-1)}{4(cn^{1/2} + 1)}.
\end{equation}
   
Next we calculate an upper bound for ${\mathcal B}$; 
if this is less than the previous expression we are done. 
Let $z_i$ be the number of times symbol $i$ appears in $D$. 
Then clearly $\sum_{i=0}^{n-1} z_i =\lceil cn^{3/2}\rceil$.

Consider the green edges in $G_{L,0,1,n}$. Let $z_0+z_1=\ell$. 
Let $m$ be the number of green edges missing in the graph; observe that $\ell/2\leq m\leq \ell$.
 We know that there are no directed cycles of green edges, otherwise we will be done by Lemma \ref{greeny}.
 Thus there are $m$ green paths; let their lengths be: $g_1$, $g_2,\dots ,g_{m}$.

%Suppose first that $m\geq 2$ and at least two of the green paths have length at least $3$. 
There are less than $3m\leq 3\ell$ vertices not belonging to a green path of length at most $3$. 
Thus there are less than $3n\ell$ black edges which are not incident with a green path of length at least $3$.
 (Note we must avoid directed cycles of black edges of length $2$, so each pair of vertices has at most one directed black edge.)
Next, there are less than  $2\sum_{i < j}(g_i+g_j)$ black edges between green paths of length at least $3$, otherwise 
there exists a directed cycle (and thus from Theorem \ref{kross} a Latin trade).

In total, the number of black edges 
in $G_{L,0,1,n}$ is less than: 
\begin{align*}
3n\ell +2\sum_{i<j}{(g_i + g_j)}  &\leq  3n\ell+2(m-1)(n-m) \\
 &\leq 3n\ell+2(\ell-1)(n-\ell/2) \\
 &< 5n\ell+\ell = (z_0+z_1)(5n+1). 
\end{align*}

Hence summing over all pairs of symbols, 
\begin{align*}
{\mathcal B}\leq (5n+1)\sum_{i\neq j}(z_i+z_j) &=   2(n-1)(5n+1)\sum_i z_i\\
&  \leq  2(n-1)(5n+1)(cn^{3/2}+1).
\end{align*}
 However, combining this upper bound for ${\mathcal B}$ with the lower bound in \eqref{blackylb} creates a contradiction for large enough $n$.
\end{proof}

We remark that the bound in the previous theorem is asymptotic. For small orders greater than $8$, linear bounds (e.g. \cite{FFR}, \cite{HAF}) are still the best known. 
To give an idea of the limitations of our methods in the main results of this paper, consider the following graph $G$ on vertex set $N(n=m^2)$. For each $k\in N(m)$, let $G_k=[km,km+1,\dots ,k(m+1)-1]$ be a directed green path in $G$. Add a directed black edge from $i$ to $j$ whenever $i-j$ is divisible by $m$ and $i<j$. Such a graph has $\Omega(n^{3/2})$ black edges yet is free of directed cycles.

\end{document}